\documentclass[12pt]{amsart}

\usepackage{amsmath}
\usepackage{amsthm}
\usepackage{amssymb}
\usepackage{mathrsfs, color}
\usepackage[colorlinks=false]{hyperref}
\usepackage{paralist}
\usepackage[]{graphicx}

\usepackage[paperwidth=216mm,paperheight=279mm,inner=3.5cm,outer=3.5cm,top=2.5cm,bottom=2.6cm]{geometry}

\theoremstyle{plain}
\newtheorem{Theorem}{Thm}[section]
\newtheorem{Thm}[Theorem]{Theorem}
\newtheorem{Lem}[Theorem]{Lemma}
\newtheorem{Cor}[Theorem]{Corollary}
\newtheorem{Prop}[Theorem]{Proposition}

\newtheorem*{Thm*}{Theorem}
\theoremstyle{definition}
\newtheorem{Def}[Theorem]{Definition}
\newtheorem{Rem}[Theorem]{Remark}
\newtheorem{Exm}[Theorem]{Example}

\newcommand\mbb{\mathbb}

\newcommand\mcal{\mathcal}

\newcommand\A{\mbb{A}}
\newcommand\C{\mbb{C}}
\newcommand\F{\mbb{F}}
\newcommand\N{\mbb{N}}
\renewcommand\P{\mbb{P}}

\newcommand\R{\mbb{R}}

\newcommand\V{\mcal{V}}
\newcommand\fL{\mathfrak{L}}

\newcommand\DP[1][n]{(\P^{#1})^\ast}
\newcommand{\ratto}{\dashrightarrow}
\newcommand\ol{\overline}
\newcommand\wh{\widehat}
\newcommand\wt{\widetilde}
\newcommand\sH{\mathscr{H}}

\newcommand\RP[1][n]{\P^{#1}(\R)}

\DeclareMathOperator\dd{D}
\DeclareMathOperator\codim{codim}
\DeclareMathOperator\rrk{rank}

\makeatletter
\renewenvironment{proof}[1][\proofname]{\par
\pushQED{\qed}%
\normalfont \topsep6\p@\@plus6\p@\relax
\trivlist
\item[\hskip\labelsep
\scshape
#1\@addpunct{.}]\ignorespaces
}{%
\popQED\endtrivlist\@endpefalse
}
\makeatother

\begin{document}
\title{Real Rank with Respect to Varieties}
\author{Grigoriy Blekherman}
\author{Rainer Sinn}

\maketitle

\begin{abstract}
We study the real rank of points with respect to a real variety $X$. This is a generalization of various tensor ranks, where $X$ is in a specific family of real varieties like Veronese or Segre varieties. The maximal real rank can be bounded in terms of the codimension of $X$ only. We show constructively that there exist varieties $X$ for which this bound is tight. The same varieties provide examples where a previous bound of Blekherman-Teitler on the maximal $X$-rank is tight. We also give examples of varieties $X$ for which the gap between maximal complex and the maximal real rank is arbitrarily large. To facilitate our constructions we prove a conjecture of Reznick on the maximal real symmetric rank of symmetric bivariate tensors. Finally we study the geometry of the set of points of maximal real rank in the case of real plane curves.
\end{abstract}

\section*{Introduction} 
The problem of decomposing a vector as a linear combination of simple vectors is central in many areas of applied mathematics, machine learning, and engineering. The length of the shortest decomposition is usually called the \textbf{rank} of the vector.

We consider the situation where the simple vectors form an \textbf{algebraic variety}. This includes well-studied and important cases such as tensor rank (real or complex), which is the rank with respect to the Segre variety, symmetric tensor rank (or Waring rank), which is the rank with respect to the Veronese variety, and anti-symmetric tensor rank, which is the rank with respect to the Grassmannian in its Pl\"ucker embedding.

We will be mostly interested in the rank over the real numbers. However, we make the definition over an arbitrary field $\F$. Let $X\subset \P^n$ be a projective variety defined over a field $\F$ 
and let $\hat{X} \subset \A^{n+1}$ be the affine cone over $X$. The variety $X$ is called nondegenerate if $X$ (or equivalently $\hat{X}$) is not contained in any hyperplane. In this case, for any 
$p\in\F^{n+1}$, $p\neq 0$, we can define the {\bf rank} of $p$ {\bf with respect to} $X$ ($X$-rank for short) as follows:
\[
  \operatorname{rank}_X(p) = \min r
  \quad
  \text{such that}
  \quad
  p=\sum_{i=1}^r x_i \quad \text{with} \quad x_1,\ldots,x_r \in \hat{X}(\F),
\]
i.e.~the rank of $p$ with respect to $X$ is the smallest length of an additive decomposition of $p$ into points of $\hat{X}$ with coordinates in $\F$. 

A rank $r$ is called \textbf{generic} if the vectors of $X$-rank $r$ contain
a Zariski open subset of $\A^{n+1}$.
Over any algebraically closed field, there is a unique generic $X$-rank for a nondegenerate variety $X$.

Over the real numbers, a rank $r$ is called \textbf{typical} if the set of vectors of $X$-rank $r$ contains an open subset
of $\R^{n+1}$ with respect to the Euclidean topology.
There can be many typical ranks for a given variety $X$. 

In this paper, we study upper bounds for the maximal real rank and the range between the minimal typical rank and the maximal real rank in terms of the variety $X$ which we allow to vary. We now outline our results:

Let $r_0$ and $r_{\max}$ denote the minimal typical rank and the maximal $X$-rank respectively. A trivial upper bound on the rank of a point is simply the dimension of the ambient vector space, which works over any field $\F$:
\begin{equation}\label{eqn:triv}
r_{\max} \leq n+1.
\end{equation}

In \cite{LandsbergTeitlerMR2628829}, Landsberg and Teitler showed that over the complex numbers (or any algebraically closed field) this bound can be improved:
\begin{equation}\label{eqn:LT}
r_{\max} \leq n-\dim X+1=\codim X+1.
\end{equation}

In \cite{BleTei}, Blekherman and Teitler showed that for real varieties
\begin{equation}\label{eqn:BT}
r_{\max} \leq 2r_0,
\end{equation}
and also 
\[
r_0=r_{\mathrm{gen}},
\]
where $r_{\mathrm{gen}}$ is the generic rank with respect to $X$ over $\C$. 

The real analogue of the Landsberg-Teitler bound is
\begin{equation} \label{eqn:bound}
r_{\max} \leq n-\dim X+2=\codim X+2.
\end{equation}
This bound can be established by adapting their proof to the real case. The proof is by a short inductive argument, so we included it below, see Theorem \ref{Thm:UpperBound}. This was already observed by Ballico, see \cite{BallMR2771116}, although his proof is terse. We give a full argument for completeness. In case that $\codim(X) + \deg(X)$ is odd, Ballico improves upon this bound and shows that $r_{\max} \leq \codim(X) + 1$.

Note that in most cases, the bound \eqref{eqn:BT} will be stronger than \eqref{eqn:bound}, unless the secant varieties of $X$ are strongly defective. However, when $X$ is a hypersurface, the Blekherman-Teitler bound is $4$, while the bound \eqref{eqn:bound} is $3$. When $X$ is a curve, its secant varieties are known to be non-defective, see e.g.~Lange \cite{LangeMR744323}, and the bound in \eqref{eqn:BT} is $n+2$ when $n$ is even, and $n+1$ when $n$ is odd. This is worse than the bound \eqref{eqn:bound}, which agrees with the trivial bound \eqref{eqn:triv}. 

We establish the tightness of the bound \eqref{eqn:bound} by constructing examples of irreducible varieties $X\subset \P^n$ of any dimension and points of maximal real $X$-rank $\codim(X)+2$. We give two constructions of such varieties, one in Section \ref{sec:bounds} and one in Section \ref{sec:Waring}. We also find examples, where this maximal real rank is typical, see Theorem \ref{Thm:tight} and Corollary \ref{Cor:tight}. The rational normal curve establishes the tightness of Ballico's improvement in the case of $\codim(X) + \deg(X)$ being odd mentioned above, see Remark \ref{Rem:Ballicobound}.

Our second contribution is a proof of a conjecture of Reznick \cite[Conjecture 4.12]{BruceMR3156559} which states that a bivariate form of degree $d$ has maximal real rank $d$ if and only if it has only real roots, see Theorem \ref{Thm:hyperbolic}. Under the additional assumption that the form has distinct roots, this conjecture was established in previous work of Causa-Re, Comon-Ottaviani, and Reznick, see \cite{CausaRe, CO, BruceMR3156559}. We provide an inductive argument which does not rely on the distinctness of the roots.

We use this characterization of binary forms of maximal real rank to study the gap between minimal typical rank and maximal real rank. In Theorem \ref{Thm:maxtypicalrank}, we find curves $X$ with the property that the maximal real rank is twice the generic $X$-rank, which is the maximal range by the Blekherman-Teitler bound \eqref{eqn:BT}. 

Our fourth contribution is that the gap between the maximal real rank and the maximal complex rank with respect to a variety $X$ can be arbitrarily large, see Corollary \ref{cor:gapcomplexrank}. We are not aware of any instances where this difference had been observed to be larger than $1$. 

The proof of Reznick's conjecture gives us one of the few instances where the set of points of maximal real rank is fully understood. Our final contribution is a study of the set of points of maximal real rank for the case of plane curves. We find that its geometry is governed by tangent lines with odd intersection multiplicity, see Theorem \ref{Thm:PlaneCurves}.

\subsection*{Acknowledgements}

We would like to thank Alessandra Bernardi, Jaros\l aw Buczy\'{n}ski, Hwangrae Lee and Giorgio Ottaviani for their comments on a draft of this paper. Both authors were partially supported by NSF grant DMS-1352073. The first author was partially supported by the Simons Institute for the Theory of Computing during the Algorithms and Complexity in Algebraic Geometry Semester.




\section{Bounds on the Real Rank}\label{sec:bounds}
In this section, let $X\subset\P^n$ be an irreducible nondegenerate variety with a smooth real point. This implies that the set of real points of $X$ is Zariski-dense in $X$, see \cite[Section 2.8]{BCRMR1659509}. We give an upper bound on the real rank of any point in $\RP$ in terms of the codimension of $X$ and establish its tightness via an explicit construction.

The following theorem is proved by adapting the Landsberg-Teitler proof to the real case. This was also observed by Ballico, see \cite[Theorem 1(i)]{BallMR2771116}.
\begin{Thm}\label{Thm:UpperBound}
The real rank of any point $p\in\RP$ with respect to $X$ is at most
\[
\rrk_X(p)\leq n-\dim(X)+2 = \codim(X)+2.
\]
\end{Thm}

\begin{proof}
The proof is by induction on the dimension of $X$.
If $\dim(X) = 1$, we have $\codim(X)+2 = n+1$, which is a trivial bound on the $X$-rank of any point in $\P^n$. Now suppose $\dim(X) = d \geq 2$ and let $p\in\P^n$ be a point, $p\notin X$. Let $H\subset \P^n$ be a generic hyperplane containing $p$. We first show that $X\cap H$ is irreducible and nondegenerate in $H$.

To show that $X\cap H$ is irreducible, we follow the argument in Griffiths-Harris \cite{GriffithsHarrisMR1288523} on page 174: Let $x\in X$ be a general point and $L\subset \P^n$ be a linear space of dimension $n-2$ which contains $p$ and intersects $X$ transversally in $x$. Let $Z$ be the unique irreducible component of $X\cap L$ containing $x$. We consider the pencil $\{H_\lambda\}\subset \DP$ of hyperplanes that contain $L$. Then $Z\subset H_\lambda \cap X$ for all $\lambda$ and $x$ is contained in a unique irreducible component $Z_\lambda$ of $H_\lambda\cap X$, because $H_\lambda$ and $X$ also intersect transversally in $x$. Now set $X' = \bigcup_\lambda Z_\lambda$. Then $\dim(X') = \dim(X)$ and therefore $X'$ is Zariski-dense in $X$. Since $H_\lambda \cap X = H_\lambda\cap X'$ for generic $\lambda$, we conclude $H_\lambda \cap X = Z_\lambda$ is irreducible for generic $\lambda$. 

To show that $X\cap H$ is nondegenerate, we follow Landsberg-Teitler \cite[Proposition 5.1]{LandsbergTeitlerMR2628829}: Suppose $H$ intersects $X$ transversally and let $H = \V(\ell)\subset \P^n$. If $X\cap H$ were contained in a hyperplane in $H$, say $X\cap H\subset \V(\ell,\ell')$, then $\ell'/\ell$ is a regular function on $X$ because it has no poles (whenever $\ell(x) = 0$ for a point $x\in X$, we have $\ell'(x) = 0$). Since $X$ is projective, $\ell'/\ell$ is constant on $X$, which means $\ell'/\ell = 0$ on $X$. This implies that $X\subset\V(\ell')$ is degenerate, which contradicts our assumptions on $X$. 

So, given a real point $p\in\P^n(\R)$, we can find a real hyperplane $H\subset\P^n$ containing $p$ such that $X\cap H$ is irreducible and nondegenerate. We can further assume that $X\cap H$ has a regular real point, because this is true for hyperplanes in at least one connected component of the complement of $X^\ast\cap p^\perp\subset(\P^n)^\ast$, the intersection of the dual variety of $X$ with the hyperplane of all hyperplanes containing $p$. So by induction, the real rank of $p$ with respect to $X\cap H$ is bounded by
\[
\rrk_{X\cap H}(p) \leq \codim_H(X\cap H) + 2 = \codim(X)+2. \qedhere
\]
\end{proof}

\begin{Rem}
The above theorem remains true if we substitute the assumption that $X$ is irreducible by the assumption that for every irreducible component of $X$, the real points are dense in it.
\end{Rem}

Our next goal is to establish the tightness of this bound by an explicit construction.
We make some observations of on the behavior of real rank under projections and joins of varieties:

\begin{Prop}\label{Prop:projections}
Let $\pi\colon \P^m \ratto \P^n$ be a linear projection with center $L$ defined over $\R$. Let $X\subset \P^m$ be an irreducible nondegenerate variety and suppose $\ol{X} = \pi(X) \subset \P^n$ is a proper subvariety. Let $\ol{p}\in\RP$ be a real point. Then 
\[
\rrk_{\ol{X}}(\ol{p}) \leq \min\{ \rrk_X(p) \colon p\in\RP[m], \; \pi(p) = \ol{p}\}.
\]
If for every real point $\ol{x}\in \ol{X}(\R)$ there is a real point $x\in X(\R)$ such that $\pi(x) = \ol{x}$, then the above inequality is an equality.
\end{Prop}

\begin{proof}
Let $p\in\RP[m]\setminus L$ be a point and $p = \sum_{i = 1}^r x_i$ with $x_i\in X(\R)$. Then $\pi(p) = \sum_{i=1}^{r'} \pi(x_i)$ and $\pi(x_i) \in \ol{X}(\R)$, where we assume, after reordering, that $x_1,\ldots,x_{r'}\in X(\R)\setminus L$ and $x_{r'+1},\ldots,x_r\in X(\R)\cap L$. Note that $r'\geq 1$. So the rank of $\pi(p)$ is at most $r'\leq r$, which shows the claimed inequality.
Now suppose that every real point $\ol{x}\in \ol{X}(\R)$ has a real preimage. Then, given $\ol{p} = \sum_{i=1}^r \ol{x_i}$ with $\ol{x_i}\in \ol{X}(\R)$, we set $p = \sum_{i=1}^r x_i$, where the $x_i\in X(\R)\setminus L$ are real preimages of $\ol{x_i}$. Then the real $X$-rank of $p$ is at most $r$ and $\pi(p) = \ol{p}$. This gives the reverse inequality.
\end{proof}

\begin{Prop}\label{Prop:joins}
Given two irreducible, nondegenerate real varieties $X\subset \P(V)$ and $Y\subset \P(W)$, we form the join $J(X,Y) \subset \P(V\oplus W)$. The real $J(X,Y)$-rank of any point $p = (x:y) \in \P(V\oplus W)$ is
\[
\rrk_{J(X,Y)} (p) = \max\{\rrk_X(x), \rrk_Y(y)\},
\]
the maximum of the real ranks of the projections of $p$ onto $\P(V)$ and $\P(W)$ with respect to $X$ and $Y$, respectively.
\end{Prop}

\begin{proof}
The inequality $\rrk_{J(X,Y)}(x:y)\geq \max\{\rrk_X(x),\rrk_Y(y)\}$ follows from Proposition \ref{Prop:projections} because the projection of $J(X,Y)$ onto $\P(V)$ is $X$ and the projection of $J(X,Y)$ onto $\P(W)$ is $Y$.

Conversely, given decompositions of $x = \sum_{i=1}^r x_i\in\P(V)$ and $y = \sum_{j=1}^s y_j\in\P(W)$ with $x_i\in X(\R)$ and $y_j\in Y(\R)$, we get a decomposition of $p = (x:y)\in \P(V\oplus W)$ by writing
\[
p = (x:y) = \sum_{i = 1}^r (x_i:y_i) = \sum_{i=1}^r 2 (\frac12 x_i: \frac12 y_i),
\]
where we assumed without loss of generality that $r\geq s$ and set $y_j = 0$ for $j>s$. Note that $(\frac12 x_i: \frac12 y_i)$ is a real point of $J(X,Y)$ because it lies on the line between $(x_i:0)\in X$ and $(0:y_i)\in Y$. So the real $J(X,Y)$-rank of $p$ is at most the maximum of the real $X$-rank of $x$ and the real $Y$-rank of $y$.
\end{proof}

We also need the following criterion certifying that a point has maximal real rank with respect to $X$.
\begin{Prop}\label{Prop:rankchar}
Let $p\in\P^n(\R)$ and $c = \codim(X)$. Assume that every real linear space $L \subset \P^n$ of dimension $c$ and containing $p$ intersects $X$ in at most $c$ real points. Then the real $X$-rank of $p$ is $c+2$, i.e.~maximal.
\end{Prop}

\begin{proof}
We prove the contrapositive: Let $r = \rrk_X(p)$ with $r\leq \codim(X)+1$ and set $c = \codim(X)$. Then $p$ lies in the linear span of $r$ real points $x_1,\ldots,x_r$ on $X$, i.e.
\[
p \in \langle x_1,x_2,\ldots,x_r\rangle.
\]
By adding general real points $x_{r+1},\ldots,x_{c+1}$, if necessary, we find a linear space $L = \langle x_1,x_2,\ldots,x_{c+1}\rangle$ containing $p$ such that $\dim(L) =  c$ and $|L\cap X(\R)| \geq c+1$.
\end{proof}

\begin{Rem}
The preceding proposition reads as follows in case that $X$ is a curve: If every hyperplane through $p$ intersects $X$ in at most $n-1$ real points, the real $X$-rank of $p$ is $n+1$.
\end{Rem}
Our general idea is to prove that a real point $p\in \P^n(\R)$, $p \notin X$, has maximal real rank $\codim(X)+2$ with respect to $X$ by considering the projection $\pi_p$ of $X$ from $p$ and using Proposition \ref{Prop:rankchar}. 
The image of the variety $X\subset \P(V)$ under $\pi_p$ is closed because $p\notin X$. We denote it by $\ol{X} = \pi_p(X)$. The assumption of Proposition \ref{Prop:rankchar} that every linear space $L\subset \P(V)$ containing $p$ intersects $X$ in at most $\codim(X)$ many real points translates by projection to the assumption that every linear space $\ol{L}\subset \P(V/p)$ of dimension $\codim(\ol{X})$ intersects $\ol{X}$ in at most $\codim(X) = \codim(\ol{X})+1$ many real points. This condition can be realized by curves of minimal degree, which are exactly the rational normal curves.

The condition that the intersection with a linear subspace of complimentary dimension has at most $\codim \ol{X}+1$ points is closely related to the variety $\ol{X}$ being \textit{of minimal degree}. However, varieties of minimal degree of higher dimension, except for irreducible quadric hypersurfaces, do not satisfy the condition that the intersection with any linear subspaces of complimentary dimension is $0$-dimensional. 
See \cite[Theorem 1]{EisenbudHarrisMR927946} for a classification of varieties of minimal degree.

We now prove that the upper bound in Theorem \ref{Thm:UpperBound} is tight, first in the case of curves. 
\begin{Thm}\label{Thm:tight}
Let $n\geq 2$.
There is an irreducible nondegenerate curve $X\subset \P^n$ with a regular real point and a real point $p\in\P^n(\R)$ such that $\rrk_X(p) = \codim(X)+2 = n+1$. Furthermore, the maximal real rank can be typical.
\end{Thm}

\begin{proof}
For $n=2$, we can take $X = \V(x_2^2x_0-(x_1^2+x_0^2)(x_1-x_0))\subset\P^2$, see also section \ref{sec:planecurves}. So let $n\geq 3$.
We choose coordinates in such a way that $p = (1:0:0:\ldots:0)$ and set $s\colon \P^{n-1}\to\P^n$ to be $(x_1:x_2:\ldots:x_n)\mapsto (0:x_1:\ldots:x_n)$.
Let $\ol{X}\subset\P^{n-1}$ be a curve of minimal degree, i.e.~$\deg(\ol{X}) = \codim(\ol{X})+1$. Let $\wh{X}$ be the cone over $s(\ol{X})\subset \P^n$ with cone point $p$.
Consider the linear system $\fL$ of cubic hypersurfaces $J(W,L)$, where $W\subset\P^2 = \V(x_3,\ldots,x_n)\subset\P^n$ is a plane cubic curve and $L = \V(x_0,x_1,x_2)\subset\P^n$. Note that $\fL$ is the set of all cubic hypersurfaces in $\P^n$ whose equations involve only the variables $x_0,x_1,x_2$. 
The intersection of a general element of this linear system with $\wh{X}$ is irreducible by Bertini's Theorem, see \cite[Th\'eor\`eme 6.3(4)]{JouMR725671} (use $f\colon X\to \P^9$, $(x_0:x_1:\ldots:x_n) \mapsto (x_0^3:x_0^2x_1:x_0^2x_2:\ldots:x_2^3)$). In $\fL$, the set of pseudohyperplanes, i.e.~hypersurfaces $S$ such that the set of real points $S(\R)$ has only one connected component and $\P^n(\R)\setminus S(\R)$ is connected, has non-empty interior in the euclidean topology, because it is true for plane cubics (e.g.~$x_2^2x_0 - (x_1^2+x_0^2)(x_1-x_0)$ is an interior point).
So set $X = \wh{X}\cap S$ for a general cubic $S\in\fL$. Note that $X$ is nondegenerate.

By genericity, the projection $\pi_p\vert_{X}\colon X \to \P^{n-1}$ induces a $1-1$ correspondence of real points of $X$ and $\pi_p(X)$. Indeed, this is true for $S = \V(x_2^2x_0 - (x_1^2+x_0^2)(x_1-x_0))$ (see Example \ref{Exm:curveP3}), so also for a general cubic in $\fL$.

Now $\dim(X) = \dim(\ol{X})$ and, by construction, every hyperplane containing $p$ intersects $X$ in at most $\codim(\ol{X}) + 1 = \codim(X)$ real points. By Proposition \ref{Prop:rankchar}, $p$ has real rank $\rrk(p) = \codim(X) + 2 = n+1$ with respect to $X$, which proves the first claim.

Note that in this example, points in a neighbourhood of $p$ have the same real rank with respect to $X$: 

By construction, every real hyperplane $H\subset\P^n$ containing $p$ intersects $X$ in at most $n-1$ many real points, counting multiplicities. By upper semi-continuity of the intersection multiplicity (see the following Lemma \ref{Lem:intersectionmult}), the same is true for real points in a neighbourhood of $p$. So by Proposition \ref{Prop:rankchar}, they also have real $X$-rank $n+1$.
So the maximal real rank is typical in these examples.
\end{proof}

We used upper semi-continuity of the intersection multiplicity in the following form in the above proof.
\begin{Lem}\label{Lem:intersectionmult}
Let $X\subset \P^n$ be a nondegenerate irreducible curve and let $(H_j)_{j\in\N}\subset (\P^n)^\ast$ be a sequence of hyperplanes in $\P^n$, converging in the euclidean topology to $H \in(\P^n)^\ast$. Suppose the length of the schemes $X\cap H_j$ is at least equal to $k\in\N$. Then the length of $X\cap H$ is at least $k$.
In particular, the number of real intersection points is upper semi-continuous (counting with multiplicities).
\end{Lem}

\begin{proof}
First note that the intersection of $X$ with any hyperplane is $0$-dimensional. The length of a $0$-dimensional scheme is equal to the constant coefficient of the Hilbert polynomial of its defining ideal. So the length of the scheme $X\cap H_j$ is the dimension of $\R[x_0,\ldots,x_n]/(I_d + \langle \ell_j\rangle_d)$, where $I$ is the homogeneous ideal defining $X$ and $H_j = \V(\ell_j)$, for large enough $d\in\N$. Choose the scaling of the linear functionals $\ell_j$ such that the sequence $(\ell_j)_{j\in\N}$ converges to $\ell$. Then $H = \V(\ell)$ and the codimension of $I_d + \langle \ell \rangle_d$ is at least $k$, which is a lower bound on the codimension of $I_d + \langle \ell_j \rangle_d$, by semi-continuity of the dimension in linear algebra.\\
This implies that the number of real intersection points is upper semi-continuous because the condition of having a given number of complex roots is open.
\end{proof}

\begin{Exm}\label{Exm:curveP3}
We construct a curve $X\subset \P^3$ such that the point $P = (1:0:0:0)$ has real rank $4$ with respect to $X$ following our proof of Theorem \ref{Thm:tight}. We start with a curve $\ol{X}$ of minimal degree in $\P^2 = \V_+(x_0) \subset \P^3$, say $\ol{X} = \{(s^2:st:t^2)\colon (s:t)\in\P^1\}\subset \P^2$. Then $\wh{X}$ is the cone over $\ol{X}$ with cone point $P$, i.e.~$\wh{X} = \V(x_1 x_3 - x_2^2)\subset \P^3$. We now intersect with the pseudohyperplane $S = \V(x_2^2 x_0 - (x_1^2 + x_0^2)(x_1-x_0))$ and set $X = \wh{X}\cap S$. Then $X$ is irreducible and the projection $\pi_P$ from $P$ induces a $1-1$ correspondence on real points. Indeed, the discriminant of the cubic $(x_2^2 x_0 + (x_1^2 + x_0^2)(x_1-x_0)$ in $x_0$ is 
\[
 -16 x_1^6 + 8 x_1^4 x_2^2 - 11 x_1^2 x_2^4 - 4 x_2^6 = 
 -\left( (2 x_2^3)^2 + 11 (x_2^2 x_1 - \frac{4}{11} x_1^3)^2 + \frac{160}{11} x_1^6 \right).
\]
So whenever $x_1$ and $x_2$ are real, the discriminant is a negative number as certified by the sum of squares representation. This means that there is a single real root $x_0$, given real values for $x_1$ and $x_2$. If $x_1$ and $x_2$ are real, then so is $x_3$ because $\pi_P(X) = \ol{X} = \V(x_1 x_3 - x_2^2)\subset \V(x_0)$. This shows that $\pi_P$ induces a $1-1$ correspondence on real points of $X$ and $\ol{X}$.

By construction, every plane $L$ in $\P^3$ through $P$ intersects $X$ in at most $2$ real points because the line $\pi_P(L)$ intersects $\ol{X}$ in at most two real points. By Proposition \ref{Prop:rankchar}, the point $P$ has real rank $4$ with respect to $X$.
\end{Exm}

We can also get rid of the genericity assumptions in the proof of Theorem \ref{Thm:tight} at the cost of taking a Zariski closure of real intersection points, i.e.~computing a real radical ideal. 
\begin{Rem}
As in the proof of Theorem \ref{Thm:tight}, let $\ol{X}\subset\P^{n-1} = \V(x_0)\subset\P^n$ be a curve of minimal degree. Set $\wh{X} = J(\ol{X},\{p\})$ and $S = \V(x_1^2x_2 - (x_0^2 + x_2^2)(x_0-x_2))\subset\P^n$. Let $L = \V(x_0,x_1,x_2)$ and suppose $L\cap \ol{X} = \emptyset$, which can be achieved by choosing suitable coordinates on $\V(x_0)$. Then the Zariski closure $X$ of $(\wh{X}\cap S)(\R)$ is irreducible and the real rank of $p = (1:0:\ldots:0)$ with respect to $X$ is $n+1$.

This is true, because $\wh{X}$ and $S$ intersect transversally in real points. Note that every intersection point of $\wh{X}$ and $S$ is regular on both varieties. So this claim is true because every real line through $p$ intersects $S$ in a simple real point (every line through $(0:1:0)$ intersects $x_2^2x_0 = (x_1^2+x_0^2)(x_1-x_0)$ transversally in one real point.
Because $\ol{X}(\R)$ is connected and the projection $\pi_p\colon X \to \ol{X}$ gives a $1-1$ correspondence on real points, $(\wh{X}\cap S)(\R)$ is also connected. These points are regular points on $\wh{X}\cap S$, so the Zariski closure is irreducible.
\end{Rem}

\begin{Cor}[to Theorem \ref{Thm:tight}]\label{Cor:tight}
For every pair of integers $d,c\geq 1$, there is an irreducible, nondegenerate variety $X\subset \P^{d+c}$ of dimension $d$ with regular real point and a real point $p\in\RP[d+c]$ with maximal real rank $c+2$ with respect to $X$. Furthermore, the maximal real rank can also be typical.
\end{Cor}

\begin{proof}
We obtain these varieties by constructing a curve $\ol{X}\subset\P^{c+1}$ of codimension $c$ with the required properties. We then take cones over $\ol{X}$ iteratively until we arrive at the desired dimension $d$. In other words, we take the join of $\ol{X}$ and a complimentary linear subspace of dimension $d-2$ as in Proposition \ref{Prop:joins}.
\end{proof}

\section{Minimal Typical Versus Maximal Rank for Curves}\label{sec:Waring}
The maximal real $X$-rank is at most twice the minimal typical real rank, which is equal to the generic complex rank, see \cite[Theorem 1]{BleTei}. Our goal is to construct explicit examples of curves where this gap is achieved. These examples will be projections of the rational normal curve from suitably chosen points. We will talk about the rank with respect to the rational normal curve in terms of the Waring rank of binary forms (equivalently symmetric tensor rank of bivariate tensors): We identify $\P^d$ with the space $\C[x,y]_d$ of binary forms of degree $d$ by associating the form 
\[
\sum_{i=0}^d \binom{d}{i}a_ix^{d-i}y^{i}
\]
to the point $(a_0:a_1:\ldots:a_d)\in \P^d$. In this identification, the rational normal curve in $\P^d$ is sent to the set of powers of linear forms. More precisely, the point $v_d(s:t)=(s^d:s^{d-1}t:\ldots:t^d)$ is identified with the binary form $(sx+ty)^d$. Via this identification, the real rank of a point $p\in\RP[d]$ with respect to the rational normal curve is equal to the real Waring rank of the associated binary form, i.e.~the smallest $r$ such that the form is a sum of $r$ powers of order $d$ of real linear forms.

The advantage is that we can characterize the forms of maximal real rank in terms of their zeros. We need some preparations.
\begin{Def}
We call a binary form $f\in\R[x,y]_d$ hyperbolic if all its roots are real, i.e.~it splits into linear factors over $\R$.
\end{Def}

Reznick has shown \cite[Corollary 4.11]{BruceMR3156559} that a hyperbolic binary form has real Waring rank $d$, which is maximal. 
An almost complete converse of this statement was established by \cite{CausaRe, CO}, who showed that any form of rank $d$ with distinct roots has all real roots. We will prove the full converse, see \cite[Conjecture 4.12]{BruceMR3156559} via a simple inductive argument:
\begin{Thm}\label{Thm:hyperbolic}
Let $f\in\R[x,y]_d$ be a binary form of degree $d\geq 3$ and suppose that $f$ is not a $d$-th power of a real linear form. The real Waring rank of $f$ is $d$ if and only if $f$ is hyperbolic.
\end{Thm}

For the proof of this Theorem, we establish several useful facts about hyperbolic binary forms and state applications of the Apolarity Lemma. Before we do so, let us observe the following fact.

\begin{Rem}\label{Rem:Ballicobound}
Let $X\subset \P^d$ be the rational normal curve of degree $d$. Then $\codim(X) + \deg(X) = 2d-1$ is odd and the maximal real rank $d$ with respect to $X$ is achieved by every hyperbolic polynomial. This maximal rank agrees with Ballico's upper bound $\codim(X) + 1$ in \cite[Theorem 1(ii)]{BallMR2771116}.
\end{Rem}

\begin{Def}
Let $f$ and $g$ be two real binary hyperbolic forms of degree $d$. We say that $f$ and $g$  interlace if there is a root of $g$ in every arc between two roots of $f$ in $\P^1(\R)\cong S^1$. If $f$ and $g$ have common roots, then this condition is understood to mean that after dividing by the greatest common divisor the quotients $\bar{f}$ and $\bar{g}$ interlace in the sense above. 

\end{Def}

\begin{Prop}[{\cite[Proposition 1.35]{Fisk}}]
\label{Prop:interlace}
Let $f$ and $g$ be two real binary hyperbolic forms of degree $d$. Then $f$ and $g$ interlace if and only if $\alpha f + \beta g$ is hyperbolic for every $\alpha,\beta\in\R$. 
\end{Prop}

We need the following characterization of hyperbolicity for binary forms. The same statement for forms with distinct roots is proved in \cite[Theorem 1]{CausaRe}.
\begin{Prop}\label{Prop:hyperbolic}
A binary form $f\in\R[x,y]_d$ of degree $d\geq 3$ is hyperbolic if and only if all its directional derivatives $\dd_v f = \langle \nabla f,v\rangle$, $v\in\R^2$, are hyperbolic.
\end{Prop}

The proof will rely on a homotopy argument. We first observe the following.
\begin{Rem}\label{Rem:zeros}
Let $f$ be a real binary form of degree $d$.
\begin{enumerate}[(a)]
\item If $\partial_x f(a,b) = 0$ and $\partial_y f(a,b) = 0$, then $f(a,b) = 0$ because $d\cdot f = x \partial_x f + y \partial_y f$.
\item Suppose $f(a,b) = 0$, $\partial_x f(a,b) = 0$, $(bx-ay)^m$ divides $\partial_x f$, and $b\neq 0$. Then $(bx-ay)^{m+1}$ divides $f$, which follows from Taylor expansion of $f(x,1)$ around $a/b$. 
\end{enumerate}
\end{Rem}

\begin{proof}[Proof of Proposition \ref{Prop:hyperbolic}]
Let $f$ be a hyperbolic form and $v\in\R^2$. After a change of coordinates and dehomogenizing, we can assume that $f$ is a polynomial in one variable of degree $d$ with only real roots and that the directional derivative $\dd_v f$ is the usual derivative $f'$. Then the derivative has $d-1$ real roots by Rolle's Theorem; so it is hyperbolic.

Conversely, suppose all directional derivatives $\dd_v f$ have only real roots and $f$ has at least one complex root. First note that all complex roots of $f$ must be simple because a double root of $f$ is a root of every directional derivative $\dd_v f$. We will perturb $f$ to arrive at a contradiction: First, we show that $f$ can have at most one complex conjugate pair of complex roots. After a change of coordinates, if necessary, let $(z_1,1)$ and $(z_2,1)$ be two distinct complex roots of $f$ with $z_1,z_2\in \C$ and $\ol{z_1}\neq z_2$. Then $f = (x- z_1 y)(x-\ol{z_1}y)(x-z_2 y)(x-\ol{z_2}y)\wt{f}(x,y)$. Assume that roots $(x,y)$ of $\wt{f}$ satisfy $y\neq 0$. Let
\[
p_t = (x-(ti + (1-t)z_1) y)(x-(ti + (1-t)z_2 y)
\]
be the binary form with roots $(ti + (1-t)z_1,1)$ and $(ti + (1-t)z_2,1)$ and set $f_t = p_t \ol{p_t} \wt{f}(x,y)$. Then $f_0 = f$ and $f_1$ has a double root at $i$ and $-i$. All directional derivatives of $f_0$ have only real roots, which means that $\partial_x f$ and $\partial_y f$ interlace. The same is true for the directional derivatives of $f_t$ for all $t\in[0,1]$: The partial derivatives $\partial_x f_t$ and $\partial_y f_t$ interlace because their zeros cannot come together. If they did, this would force a multiple real root of $f$ by Remark \ref{Rem:zeros} because all roots $(x,y)$ of $f_t$ satisfy $y\neq 0$. Now $f_1$ has a double complex root and $\partial_x f_1$ and $\partial_y f_1$ have only real roots, a contradiction.

Next we show that $f$ cannot have a pair of complex conjugate roots by the same idea: After a change of coordinates, we can factor $f$ as
\[
f = (x-z_1 y)(x-\ol{z_1}y)(x-a_1y)\ldots(x-a_{d-2}y),
\]
where $z_1\in \C$ and $a_1,\ldots,a_{d-2}\in\R$. We use the homotopy
\[
f_t = (x-(ti + (1-t)z_1) y)(x-(-ti + (1-t)\ol{z_1}) y)(x-(1-t)a_1 y)\ldots (x-(1-t)a_{d-2}y).
\]
Then $f_0 = f$ and $f_1= x^{d-2}(x^2+y^2)$. Since the partial derivatives of $f$ interlace, the same is true for the partial derivatives of $f_1$ by Remark \ref{Rem:zeros} by the same argument as before. In particular, all their roots have to be real, but this is false, because
\[
\partial_x f_1 = x^{d-3}(dx^2 + (d-2)y^2)
\]
has a pair of complex conjugate roots ($d\geq 3$). So we have shown that all roots of $f$ are real, given that all its directional derivatives have only real roots.
\end{proof}

We need the following application of the Apolarity Lemma, see \cite[Chapter I]{IarKanMR1735271}.
\begin{Lem}\label{Lem:apolar}
Let $h\in\R[x,y]_{d-1}$ be a binary form of degree $d-1\geq 2$. Suppose that the real Waring rank of $h$ is at most $d-2$ and that every decomposition of $h$ into the sum of at most $d-2$ powers of linear forms uses the power $y^{d-1}$. Then $h$ is equal to $y^{d-1}$ up to scaling or $h = a y^{d-1} + b \ell^{d-1}$ for some $a,b\in \R$ and a linear form $\ell\in\R[x,y]$.
\end{Lem}

\begin{proof}
The apolar ideal $h^\perp\subset\R[x,y]$ of $h$ is generated by two polynomials $r_1,r_2\in\R[x,y]$ such that $\deg(r_1) + \deg(r_2) = d+1$, see \cite[Theorem 1.44(iv)]{IarKanMR1735271}. We assume $\deg(r_1)\leq \deg(r_2)$. The fact that every decomposition of $h$ as the sum of at most $d-2$ powers of real linear forms uses $y^{d-1}$ is equivalent to the fact that every polynomial in $(h^\perp)_{d-2}$ with only real roots is divisible by $x$ by the Apolarity Lemma, \cite[Lemma 1.31]{IarKanMR1735271}.
The set of hyperbolic polynomials has non-empty interior in $\R[x,y]_{d-2}$ (see \cite{NuijMR0250128}), which implies that every polynomial in $(h^\perp)_{d-2}$ is divisible by $x$. In particular, $\deg(r_2)\geq d-1$, because $r_1$ and $r_2$ have no common zeros. 
This leaves the two cases $\deg(r_1)=1$, so that $\deg(r_2) = d$, or $\deg(r_1) = 2$, so that $\deg(r_2) = d-1$.
In case $\deg(r_1) = 1$, $r_1$ must be equal to $x$ up to scaling. By the Apolarity Lemma, $h$ has real Waring rank $1$ in this case and is equal to $y^{d-1}$ up to scaling.
In case $\deg(r_2) = 2$, there is a linear form $\ell_\perp = u x + v y$ such that $r_1 = x \ell_\perp$. In this case, $D_{(u,v)} \partial_x h = 0$, which shows that $h$ is equal to $a y^{d-1} + b \ell^{d-1}$, where $\ell = -v x + u y$, for some $a,b\in\R$.
\end{proof}

\begin{Exm}
Consider the binary form $f = y^{d-1}(ax + by) + c \ell^d$, $d\geq 4$, for a linear form $\ell\in\R[x,y]$ and $a,b,c\in\R$. Then the real Waring rank of $f$ is at most $d-1$ or $f$ is hyperbolic: By the Structure Theorem, the apolar ideal of $f$ is generated by two polynomials
\[
f^\perp = \langle r_1,r_2 \rangle
\]
with $\deg(r_1) + \deg(r_2) = d+2$, see \cite[Theorem 1.44(iv)]{IarKanMR1735271}. Clearly, the apolar ideal contains the form $x^2\ell_\perp$, where $\ell_\perp$ is the linear form that is apolar to $\ell$. Now there are three cases: If $f^\perp$ contains a linear form, then $f = by^d$ or $c\ell^d$, i.e.~the Waring rank of $f$ is $1$.
Suppose $f^\perp$ contains a form $q$ of degree $2$, then $q$ divides $x^2 \ell_\perp$, because $\deg(r_2) = d \geq 4$. Therefore, $q = x\ell_\perp$ or $q = x^2$.
If $q = x^2$, then $f = y^{d-1}(ax+by)$ is hyperbolic. If $q = x\ell_\perp$, then the real Waring rank of $f$ is $2$ by the Apolarity Lemma.

Thirdly, if $f^\perp$ does not contain a quadric, then $r_1 = x^2 \ell_\perp$ and $\deg(r_2) = d-1\geq 3$. Now choose generic linear forms $\ell_1,\ldots,\ell_{d-4}$ such that $r = r_1 \cdot \ell_1 \cdot\ldots\cdot \ell_{d-4}$ has only distinct roots except for $(0:1)$, which is a double root of $r_1$. By perturbing $r$ with a small multiple of $r_2$, we obtain a polynomial in $f^\perp$ with distinct real roots of degree $d-1$. By the Apolarity Lemma, this implies that the real Waring rank of $f$ is at most $d-1$.
\end{Exm}

We are now ready to prove that all binary forms of maximal real Waring rank and degree at least $3$  are hyperbolic.
\begin{proof}[Proof of Theorem \ref{Thm:hyperbolic}]
We prove the statement by induction on the degree of $f$:
The base case $d=3$ was done by Reznick, see \cite[Theorem 5.2]{BruceMR3156559}.

So let $f\in\R[x,y]_d$ be a hyperbolic binary form of degree $d>3$. Then every directional derivative of $f$ is a hyperbolic form of degree $d-1$. By induction, they all have real Waring rank $d-1$. If the Waring rank of $f$ was $d-1$, then there would be a directional derivative with real Waring rank at most $d-2$. So $f$ must have real Waring rank $d$. For a different proof of this direction, see \cite{BruceMR3156559}.

On the other hand, let $f$ be a binary form of degree $d>3$ and suppose that the real Waring rank of $f$ is $d$. If, after a change of coordinates, $f = y^{d-1}\ell$ or $f = y^{d-1}(ax+by) + c \ell^d$ for some linear form $\ell\in\R[x,y]$ and $a,b,c\in\R$, then $f$ is hyperbolic (see the preceding Example for the second case).
Otherwise, every directional derivative of $f$ has real Waring rank $d-1$. Indeed, if that were false, we can assume after a change of coordinates that $\partial_x f$ has Waring rank at most $d-2$.
By Lemma \ref{Lem:apolar}, there is a decomposition
\[
\partial_x f = \sum_{i=1}^{d-2} \alpha_i \ell_i^{d-1},
\]
where the coefficient $a_i$ of $x$ in all linear forms $\ell_i = a_i x + b_i y$ is nonzero, because $\partial_x f$ is not equal to $ay^{d-1} + c\ell^{d-1}$. By taking antiderivatives, we get
\[
f = \alpha y^d + \sum_{i=1}^{d-2}\frac{\alpha_i}{d a_i} \ell_i^d,
\]
for some $\alpha\in\R$. This equation implies that the real Waring rank of $f$ is at most $d-1$. So every directional derivative of $f$, which has degree $d-1$, has real Waring rank $d-1$.
Therefore, by induction, every directional derivative of $f$ is hyperbolic. And then so is $f$, by Proposition \ref{Prop:hyperbolic}.
\end{proof}

We can use this result on the real Waring rank of binary forms to construct a curve such that the gap between the minimal typical real rank and the maximal typical rank is as large as possible. We construct this curve by projecting the rational normal curve.
\begin{Thm}\label{Thm:maxtypicalrank}
Let $p \in\R[x,y]_d$ be a binary hyperbolic form with distinct real roots. Let $X\subset\P^{d-1}$ be the projection of the rational normal curve (embedded into $\P(\C[x,y]_d)$ as the $d$-th powers of linear forms) from the point $p$. Then the minimal typical real rank with respect to $X$ is $\lceil d/2 \rceil$, whereas the maximal typical real rank with respect to $X$ is $d$.
\end{Thm}

\begin{proof}
The minimal typical real rank with respect to $X\subset\P^{d-1}$ is equal to the generic complex rank, which is $\lceil d/2\rceil$ by counting the dimension of the secant varieties of $X$. Indeed,
\[
\dim(S_{\left( \lceil d/2 \rceil -1 \right)}X) = 2\lceil d/2\rceil - 1 \geq d-1 
\]
The fact that the maximal typical real rank is $d$ follows from our above results on the Waring rank: Let $q\in\R[x,y]_d$ be any hyperbolic binary form that interlaces $p$. Then every form in the span of $p$ and $q$ is hyperbolic by Proposition \ref{Prop:interlace}. So the image of $q$ under the projection from $p$ is a real point of rank $d$ with respect to $X$ by Proposition \ref{Prop:projections} and Theorem \ref{Thm:hyperbolic}. As the set of interlacers of $p$ is open in the Euclidean topology, this gives an open set of real points in $\P^{d-1}$ of rank $d$.
\end{proof}

\begin{Rem}
In fact, in our construction in the proof of Theorem \ref{Thm:maxtypicalrank}, the set of all points of real $X$-rank $d$ is the projection of the set of all interlacing polynomials of the fixed strictly hyperbolic polynomial $p$, which is a convex cone, cf.~\cite[Corollary 2.7]{cynmardan}. In particular, the set of points of real $X$-rank $d$ is connected. By choosing a form $p$ with repeated roots, we can make the dimension of the set of points of real $X$-rank $d$ smaller; e.g.~if we set $p = x^{d-1}y$, then the dimension of bivariate forms interlacing $p$ is $2$ (projectively). 
\end{Rem}

\begin{Cor}\label{cor:gapcomplexrank}
Let $p \in \R[x,y]_d$ be a generic binary hyperbolic form with distinct real roots and let $X\subset \P^{d-1}$ be the projection of the rational normal curve as above. Then the maximal complex $X$-rank is at most $\lceil (d+1)/2 \rceil$, whereas the maximal real rank is $d$.
\end{Cor}

\begin{proof}
The generic complex rank with respect to the rational normal curve $C\subset\P^d$ is $\lceil (d+1)/2 \rceil$ by count of dimensions of the secant varieties. So if we project from a point $p$ with this rank, then the maximal rank in the image $\P^{d-1}$ is at most this rank $\lceil (d+1)/2 \rceil$ by Proposition \ref{Prop:projections}.
\end{proof}

\begin{Rem}
The condition on the polynomial to be generic in the above corollary is well understood. It is different for odd and even degrees. If $d$ is even, then the condition is that the Hankel matrix (middle Catalecticant) of $p$ is non-singular. If $d$ is odd, the condition is slightly more complicated. See Comas-Seiguer \cite[section 3]{CS} for details.
\end{Rem}

\section{Plane Curves}\label{sec:planecurves}
In case of plane curves, we have $\dim(X) = 1$ and $n=2$, so the maximal real rank of any point is $3$, the trivial bound. This bound is tight. Our goal is to show that regions of points of rank $3$ are bounded by real flex lines to the curve, given that they have non-empty interior. 

The first observation is that the degree of the curve $X$ must be odd if there is a point of real rank $3$.
\begin{Rem}
Given an irreducible nondegenerate plane curve $X\subset\P^2$ with regular real point and a point $p\in\P^2(\R)$ with $\rrk(p) = 3$, the degree of $X$ must be odd and $X(\R)$ has no ovals: If $X(\R)$ has an oval, then the set of lines through the interior of that oval cover $\P^2(\R)$ and intersect the oval in two real points. 

For the real topology of $X$, this means that $X(\R)$ is a pseudoline, i.e.~$X(\R)$ is connected and $\P^2(\R)\setminus X(\R)$ is also connected.
\end{Rem}

\begin{Thm}\label{Thm:PlaneCurves}
Let $X\subset\P^2$ be an irreducible nondegenerate curve and assume that $X(\R)$ is non-empty and contains only regular points of $X$. Let $\sH\subset\P^2(\R)$ be the union of all tangents to a real point of $X$ that meet $X$ at odd order $\geq 3$. Then the set of points in $\P^2(\R)\setminus \sH$ of real $X$-rank $3$ is a union of connected components of $\P^2(\R)\setminus \sH$ (possibly empty). In other words, the regions of points of real $X$-rank $3$ with non-empty interior are bounded by odd order tangents.
\end{Thm}

\begin{proof}
We show that the set of points of real $X$-rank $3$ in $\P^2(\R)\setminus \sH = : \sH^c$ is open and closed:

A point $p\in\RP[2]$ has real $X$-rank $3$ if and only if every line through $p$ intersects $X(\R)$ in at most $1$ point. As a condition in the line, this is open because $X(\R)$ is compact, as long as we stay away from odd order tangents. By compactness of the dual projective space $\DP[2](\R)$, this also translates into an open condition on the point $p$ by duality.

To see that the regions of points of real $X$-rank $3$ are closed in $\sH^c$, we show that the complement is open. A point $q\in \sH^c$ of real $X$-rank less than $3$ either has rank $2$ or $q\in X(\R)$. Suppose $\rrk_X(q)=2$, then there is a line through $q$ that intersects $X(\R)$ in $2$ distinct points. We can assume that this line is not tangent to $X$ at any real point (by choosing a generic line through $q$). Since $X(\R)$ has no singular points, we cover an open neighbourhood of $q$ by varying one of these $2$ points in $X(\R)$. If $q\in X(\R)$, then we can find a line intersecting $X(\R)$ transversally in $q$ and in at least $2$ more real points because $X$ has odd degree. Using the same argument as for the rank $2$ case, we see that all points in a neighbourhood of $q$ (in $\sH^c$) have real $X$-rank at most $2$.
\end{proof}
\begin{figure}[b]
\centering
\includegraphics[scale=0.8]{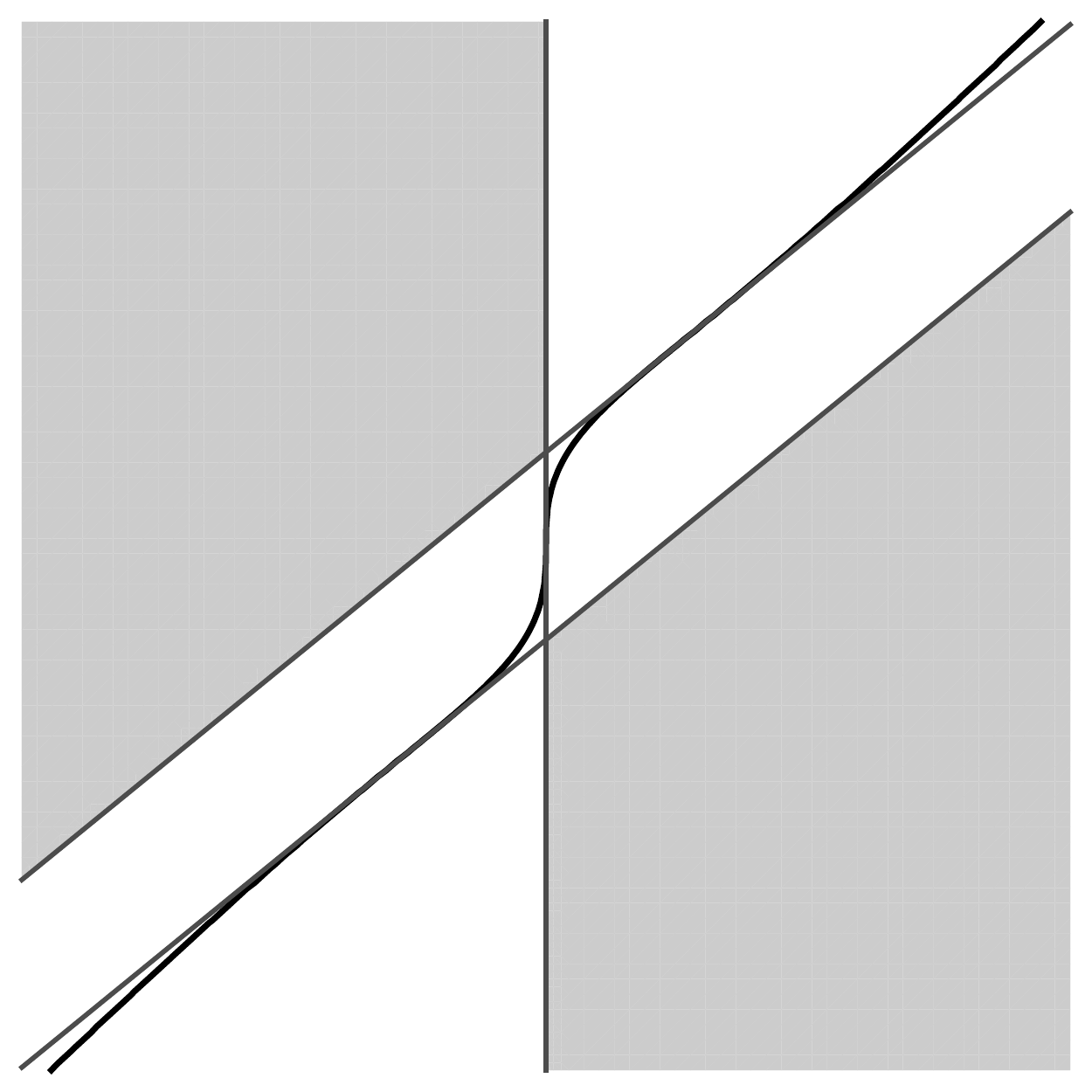}
\caption{A plane cubic, its $3$ real flex lines, and the triangular region of points of real rank $3$.}
\label{fig:deg3}
\end{figure}

\begin{Rem}
\begin{enumerate}[(a)]
\item For most curves, the only odd order tangents with an order of tangency at least $3$ are flex lines, i.e.~tangents to $X$ at inflection points. This is the case for all Pl\"ucker curves. In this case, the only singularities of the dual plane curve are nodes and cusps.
\item In particular, we have seen in the proof of the above theorem, that all regions in $\RP[2]\setminus \sH$ that meet $X(\R)$ are regions of rank $2$ points.
\end{enumerate}
\end{Rem}

\begin{Exm}
\begin{enumerate}[(a)]
\item Consider the plane cubic $X = \V(x_2^2x_0 - (x_1^2+x_0^2)(x_1-x_0))$. This cubic has only one connected component in $\RP[2]$ and three real flex lines. They bound a triangle. The points in the interior of this triangle have real rank $3$ with respect to $X$, see Figure \ref{fig:deg3}.
\item Consider the plane curve $X = \V(x_1^5-x_1^3x_2^2 - x_0(x_0^4-19/20 x_0^2 x_2^2+x_2^4))$ of degree $5$. It has $5$ real flex lines. The region of points of rank $3$ is connected with non-empty interior and is bounded by four flex lines. It is shown in Figure \ref{fig:deg5}.
\begin{figure}[h]
\centering
\includegraphics[scale=0.8]{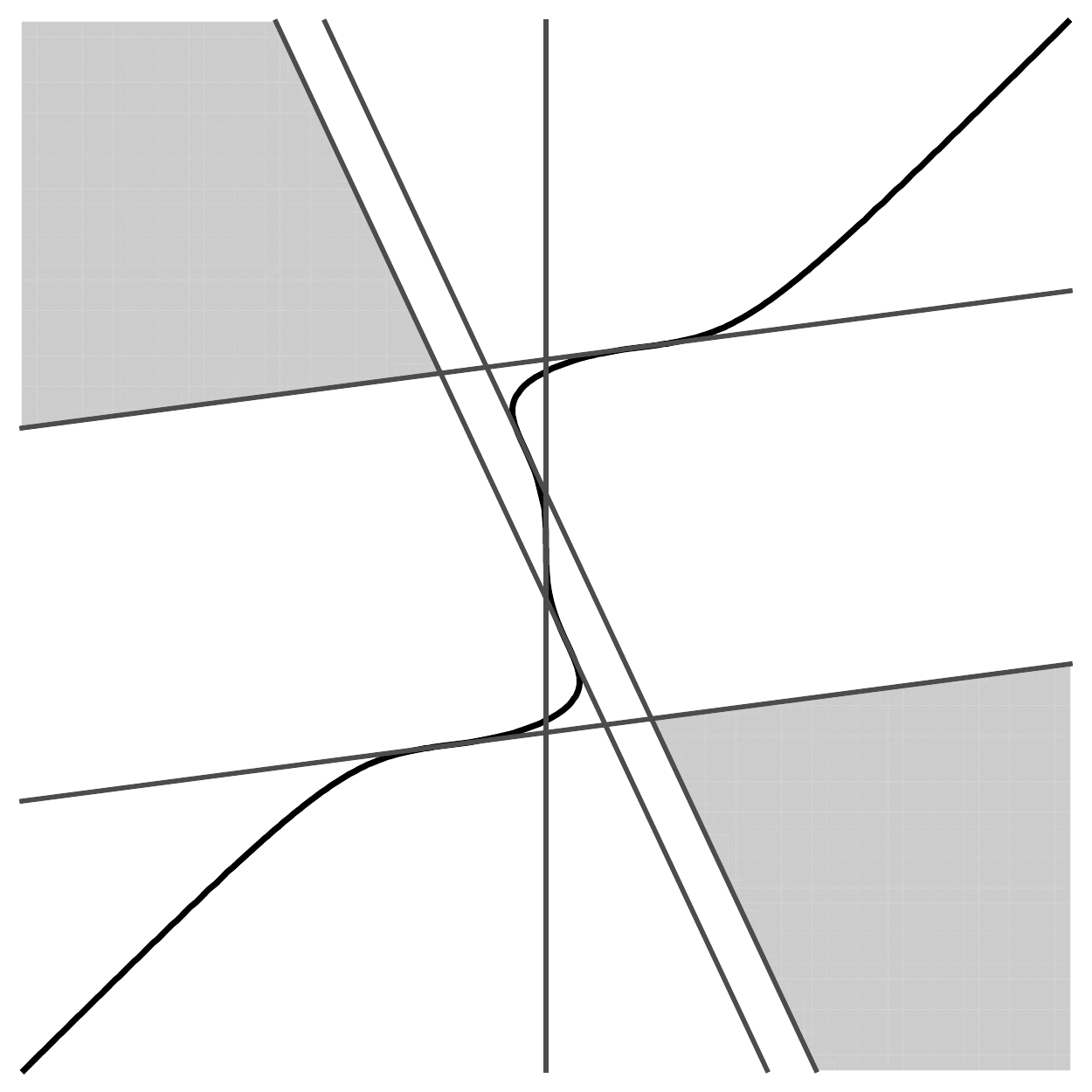}
\caption{A curve of degree $5$, its $5$ real flex lines, and the region of points of real rank $3$.}
\label{fig:deg5}
\end{figure}
\end{enumerate}
\end{Exm}

\end{document}